\documentclass{article}
\usepackage{graphicx}
\usepackage[english]{babel}
\usepackage{epsfig}
\usepackage{amsmath}
\usepackage{amssymb}
\usepackage{amsthm}
\usepackage{amsfonts}
\usepackage{color}
\newtheorem{theorem}{Theorem}[section]
\newtheorem{teo}{Theorem}%[section]
\newtheorem{definition}[theorem]{Definition}

\newtheorem{corollary}[theorem]{Corollary}

\newtheorem{example}[theorem]{Example}

\newcommand{\C}{\mathbb{C}}

\DeclareMathOperator{\IIm}{Im}
\newcommand{\HH}{\mathbb{H}}
\DeclareMathOperator{\RRe}{Re}

\newcommand{\hh}{{\mathbb{H}}}
\newcommand{\s}{{\mathbb{S}}}
\newcommand{\cc}{{\mathbb{C}}}
\newcommand{\rr}{{\mathbb{R}}}
\newcommand{\nn}{{\mathbb{N}}}
\newcommand{\zz}{{\mathbb{Z}}}

\newcommand{{\ee}}{{\`e}}
\newcommand{{\aA }}{{\`a}}
\newcommand{{\oo}}{{\`o}}
\newcommand{{\uu}}{{\`u}}
\newcommand{{\ii}}{{\`i}}

\begin{document}

\title{The Mittag-Leffler Theorem for regular functions of a quaternionic variable 
%\thanks{This project has been supported by G.N.S.A.G.A. of INdAM - Rome (Italy), by MIUR of the Italian Government (Research Projects: PRIN ``Real and complex manifolds: geometry, topology and harmonic analysis'' and FIRB ``Geometric function theory and differential geometry").}
}

\author{Graziano Gentili  \and Giulia Sarfatti}
%etc.
%}

%\authorrunning{G. Gentili - G. Sarfatti} % if too long for running head

\author{Graziano Gentili \footnote{This project has been supported by G.N.S.A.G.A. of INdAM - Rome (Italy), by MIUR of the Italian Government (Research Projects: PRIN ``Real and complex manifolds: geometry, topology and harmonic analysis'', FIRB ``Geometric function theory and differential geometry" and SIR \textquotedblleft Analytic aspects in complex and hypercomplex geometry'').}\\
\normalsize Dipartimento di Matematica e Informatica ``U. Dini'', Universit\`a di Firenze \\
\normalsize Viale Morgagni 67/A, 50134 Firenze, Italy,  gentili@math.unifi.it \\
\and Giulia Sarfatti $^*$\\
\normalsize Dipartimento di Matematica e Informatica ``U. Dini'', Universit\`a di Firenze \\
\normalsize Viale Morgagni 67/A, 50134 Firenze, Italy,  giulia.sarfatti@unifi.it }
\date{  }
\date{}

\maketitle
\begin{abstract}
We prove a version of the classical Mittag-Leffler Theorem for regular functions over quaternions. Our result relies upon an appropriate notion of principal part, that is inspired by the recent definition of spherical analyticity. 


\end{abstract}

\section{Introduction}

The class of (slice) regular functions of a quaternionic variable was introduced in \cite{CR}, \cite{GSAdvances}, and proved to be a good counterpart of the class of holomorphic functions, in the quaternionic setting.  Regular functions have nice new features, when compared with the classical quaternionic Fueter regular functions: for instance natural polynomials and power series are regular, and regular functions can be expanded in power series on special classes of domain in the space of quaternions $\HH$.

This theory is having a fast development in several directions, and is by now already well established; it has interesting applications to the construction of a noncommutative  functional calculus, \cite{librofc}, and to the classification of Orthogonal Complex Structures in subdomains of the space $\HH$, \cite{GenSalSto}. An exhaustive presentation of this theory can be found in \cite{libroGSS}. 

Many results that concern regular functions reflect the structure of their complex analogues, other are surprisingly different: for example the zero sets of regular functions (and the sets of poles of \emph{semiregular} functions) consist of isolated points and isolated 2-dimensional spheres.

One of the fundamental results in the theory of holomorphic functions is the celebrated Mittag-Leffler Theorem, that has been used in many different contexts, and in particular in that of sheaves of meromorphic functions. 

\begin{theorem}
Let $\Omega$ be an open subset of the complex plane $\C$, and let $A\subset \Omega$. Let us suppose that $A$ has no accumulation point in $\Omega$ and, for any $a\in A$, choose an integer  $m(a) \in \nn$ and a rational function 
\[
P_a(z)=\sum_{j=1}^{m(a)}(z-a)^{-j}c_{j,a}.
\]
Then there exists a meromorphic function $f: \Omega \to \C$, whose principal part at every $a\in A$ is $P_a$, having no other pole in $\Omega$.
\end{theorem}

The search for an analogous result for regular functions, connected with the under-construction theory of sheaves of regular  and  semi regular functions, \cite{sheaves}, inspired this work. Since in the new environment of regular functions there are several, non equivalent  notions of analyticity, \cite{powerseries}, \cite{spherical}, an important step is the choice of the ``right'' notion of principal part.  We adopt here the approach suggested by spherical series, \cite{spherical}, which, together with the quaternionic version of the Runge Theorem, \cite{runge}, leads to the aimed result.

\section{On quaternionic analyticity}
With the usual notations, let $\HH=\rr+\rr i+\rr j+\rr k$ denote the four dimensional non-commutative real algebra of quaternions. For any $q=x_0+x_1i+x_2j+x_3k \in \HH$ let $\RRe(q)=x_0$ and $\IIm(q)=x_1i+x_2j+x_3k $ denote its real and imaginary parts and let $|q|=\sqrt{x_0^2+x_1^2+x_2^2+x_3^2}$ be its modulus. 
The definition of regular function is given in terms of the elements of the $2$-sphere $\s = \{q \in \hh : q^2 = -1\}$ of \emph{quaternion imaginary units}.
\begin{definition}%\label{definition}
Let $\Omega$ be a domain in $\hh$ and let $f : \Omega \to \hh$ be a function. For all $I \in \s$, let us denote $L_I = \rr + I \rr$, $\Omega_I = \Omega \cap L_I$ and $f_I = f_{|_{\Omega_I}}$. 
The function $f$ is called \emph{(slice) regular} if, for all $I \in \s$, the restriction $f_I$ is holomorphic, i.e. the function $\bar \partial_I f : \Omega_I \to \hh$ defined by
$$
\bar \partial_I f (x+Iy) = \frac{1}{2} \left( \frac{\partial}{\partial x}+I\frac{\partial}{\partial y} \right) f_I (x+Iy)
$$
vanishes identically.
\end{definition}

One of the reasons of the immediate interest for regular functions stays in the fact that  an analog of Abel's Theorem holds: any power series  
$$f(q) = \sum_{n \in \nn} q^n a_n$$ 
with quaternionic coefficients $\{a_n\}$ defines a regular function on its ball of convergence 
$B(0,R) = \{q \in \hh : |q| <R\}$.
The set of such series inherits the classical multiplication $*$ defined for quaternionic polynomials (or, more in general, for polynomials with coefficients in a noncommutative ring):
\begin{equation}\label{product}
\left(\sum_{n \in \nn} q^n a_n \right)*\left( \sum_{n \in \nn} q^n b_n \right) = \sum_{n \in \nn} q^n \sum_{k = 0}^n a_k b_{n-k}.
\end{equation}
Let now $(q-q_0)^{*n} = (q-q_0)*\ldots*(q-q_0)$ denote the $*$-product of $n$ copies of $q \mapsto q-q_0$. In \cite{powerseries}  series of the form
\begin{equation}\label{regularseries}
f(q) = \sum_{n \in \nn} (q-q_0)^{*n} a_n
\end{equation}
are studied, whose sets of convergence are balls with respect to the distance $\sigma : \hh \times \hh \to \rr$ defined in the following fashion.

\begin{definition}
For all $p,q \in \hh$, we set
\begin{equation}
\sigma(q,p) = \left\{
\begin{array}{ll}
|q-p| & \mathrm{if\ } p,q \mathrm{\ lie\ on\ the\ same\ complex\ plane\ } \rr+I\rr\\
\omega(q,p) &  \mathrm{otherwise}
\end{array}
\right.
\end{equation}
where
\begin{equation}
\omega(q,p) = \sqrt{\left[Re(q)-Re(p)\right]^2 + \left[|Im(q)| + |Im(p)|\right]^2}. 
\end{equation}
\end{definition}
\noindent A new notion of analyticity can be given in tems of the distance $\sigma$:

\begin{definition}\label{defisigmaanaliticita}
If $\Omega$ is a domain in $\hh$, a function $f : \Omega \to \hh$ is called \emph{$\sigma$-analytic} if it admits at every $q_0 \in \Omega$ an expansion of type \eqref{regularseries} that is valid in a $\sigma$-ball $\Sigma(q_0,R) = \{q \in \hh : \sigma(q,q_0) <R\}$ of positive radius $R$.
\end{definition}

\noindent Regularity and $\sigma$-analyticity turn out to be the same notion, as it appears in the following result proved in \cite{powerseries}.

\begin{theorem}\label{sigmaanaliticita}
If $\Omega$ is a domain in $\hh$, a function $f : \Omega \to \hh$ is regular if and only if it is \emph{$\sigma$-analytic}.
\end{theorem}

The meaning of Theorem \ref{sigmaanaliticita} is not as strong as in the complex case, since $\sigma$-analyticity has not the features one may imagine at a first glance. In fact the topology induced by the distance $\sigma$ is finer than the Euclidean: if $q_0 = x_0+Iy_0$ does not lie on the real axis then for $R<y_0$ the $\sigma$-ball $\Sigma(q_0,R)$ reduces to a ($2$-dimensional) disc $\{z \in L_I : |z-q_0|<R\}$ in the complex plane $L_I$ through $q_0$ (see \cite{powerseries} for a presentation of the shape of $\sigma$-balls). Hence the  behavior of $f$ in a Euclidean neighborhood of $q_0$ cannot be envisaged by the series expansion \eqref{regularseries}, which, in general, will only represent $f$ along the complex plane $L_I$ containing $q_0$. 
%This curious phenomenon can be better understood if we consider how the coefficients $a_n$ in \eqref{regularseries} are constructed. Let us first recall the following definition given by Gentili and Struppa.
%\begin{definition}
%Let $f$ be a regular function on a domain $\Omega\subseteq \hh$: the \emph{Cullen (or slice) derivative} $f'$ of $f$ is defined to equal
%\begin{equation}
%\partial_I f (x+Iy) = \frac{1}{2} \left( \frac{\partial}{\partial x}-I\frac{\partial}{\partial y} \right) f_I (x+Iy)
%\end{equation}
%at each point $x+Iy \in \Omega_I$. 
%\end{definition}
%
%If $f^{(n)}$ denotes the $n$th Cullen derivative of $f$, then $f$ expands as
%$$f(q) = \sum_{n \in \nn} (q-q_0)^{*n} \frac{f^{(n)}(q_0)}{n!}$$
%at each point $q_0 \in \Omega$.
%Now, $f'(q_0)$ represents the differential at $q_0$ of $f$ restricted to $L_I$, but it cannot, in general, represent the whole differential of $f$ at $q_0$. Indeed, theorem \ref{derivative} prevents the existence of a quaternion having such a property (except for the trivial case when $f$ is affine). It is reasonable, then, that the expansion in the previous formula does not, in general, predict the behavior of $f$ in a Euclidean neighborhood of $q_0$.
To understand this phenomenon we will present what we believe to be a meaningful example (see e.g. \cite{libroGSS})

\begin{example}%\label{discontinuous}
Let $\Delta$ be the open unit disc centered at the origin of $L_i=\rr+i\rr=\cc$ and let $f:\Delta \to \cc$, $f(z)= \sum_{n \in \nn} z^n a_n$ be a holomorphic function whose maximal domain of definition is $\Delta$. Then the power series 
$$
f(q) = \sum_{n \in \nn} (q-\frac34i)^n a_n
$$
does not converge on a Euclidean neighborhood of $\frac34i$ but only in a $2$-dimensional disc of $\cc$ containing $\frac34i$.
\end{example}

As explained in  \cite{advancesrevisited}, the situation is much better if the domain $\Omega$ is carefully chosen. Consider the following class of domains:

\begin{definition}
Let $\Omega$ be a domain in $\hh$. If 
$$\Omega = \bigcup_{x+Iy\in \Omega} x+y\s$$
then $\Omega$ is called an \emph{(axially) symmetric domain}. If the domain $\Omega$  intersects the real axis and is such that for all $I \in \s$, $\Omega_I = \Omega \cap L_I$ is a domain in $L_I \simeq \cc$ then $\Omega$ is called a \emph{slice domain}.
\end{definition}

%\noindent Slice domains (and hence symmetric slice domains) take advantage of the Identity Principle:
%
%\begin{theorem}[Identity principle]%\label{identity}
%Let $\Omega$ be a symmetric slice domain and let $f,g : \Omega \to \hh$ be slice regular. Suppose that $f$ and $g$ coincide on a subset $C$ of $\Omega_I$, for some $I \in \s$. If $C$ has an accumulation point in $\Omega_I$, then $f \equiv g$ in $\Omega$.
%\end{theorem}
%

\noindent Regular functions $f$ on symmetric slice domains are affine when restricted to a single $2$-sphere $x+y\s$ (see, e.g., \cite[Theorem 3.2]{advancesrevisited}, \cite[Theorem 1.10]{spherical}).
As a consequence, if $f$ is a regular function on a symmetric slice domain then its values can all be recovered from those of one of its restrictions $f_I$.  This last fact leads to the definition of a stronger form of analyticity than the one presented in Theorem \ref{sigmaanaliticita}, which is related to a different type of series expansion valid in Euclidean open sets. If we denote as $R_{q_0}f : \Omega \to \hh$ the function such that
$$f(q) = f(q_0) + (q-q_0)*R_{q_0}f(q),$$ then the following result holds (see \cite[Theorem 4.1]{spherical}).

%\begin{proposition}
%Let $\{a_n\}_{n \in \nn}\subset \hh$ and suppose 
%\begin{equation}
%\limsup_{n \to +\infty} |a_n|^{1/n} = 1/R
%\end{equation}
%for some $R >0$. Let $q_0 = x_0+Iy_0 \in \hh$ with $x_0 \in \rr, y_0>0, I \in \s$ and set $P_{2n}(q) = [(q-x_0)^2+y_0^2]^n$ and $P_{2n+1}(q) = [(q-x_0)^2+y_0^2]^n(q-q_0)$ for all $n \in \nn$. Then the function series $\sum_{n \in \nn}P_n(q) a_n$ converges absolutely and uniformly on compact sets in 
%\begin{equation}
%U(x_0+y_0\s,R) = \{q \in \hh : |(q-x_0)^2+y_0^2| < R^2\},
%\end{equation}
%where it defines a regular function. Furthermore, the series diverges at every $q \in \hh \setminus \overline{U(x_0+y_0\s,R)}$.
%\end{proposition}

\begin{theorem}
Let $f$ be a regular function on a symmetric slice domain $\Omega$, and let $x_0,y_0 \in \rr$ and $R >0$ be such that $$U(x_0+y_0\s,R)=\{q \in \hh : |(q-x_0)^2+y_0^2| < R^2\} \subseteq \Omega.$$ For all $q_0 \in x_0+y_0\s$, setting 
$$A_{2n} =  (R_{\bar q_0}R_{q_0})^{n}f(q_0)$$ and $$A_{2n+1} = R_{q_0}(R_{\bar q_0}R_{q_0})^{n}f(\bar q_0),$$ we have that
\begin{equation}\label{expansion2}
f(q) = \sum_{n \in \nn}[(q-x_0)^2+y_0^2]^n [A_{2n} + (q-q_0)A_{2n+1}]
\end{equation}
for all $q \in U(x_0+y_0\s,R)$.
\end{theorem}

\noindent Here is the announced notion of analyticity, \cite{spherical}.

\begin{definition}
Let $f$ be a regular function on a symmetric slice domain $\Omega$. We say that  $f$ is \emph{symmetrically analytic} if it admits at any $q_0 \in \Omega$ an expansion of type \eqref{expansion2} valid in a Euclidean neighborhood of $q_0$.
\end{definition}

\noindent Thanks to the previous theorem, we obtain:

\begin{corollary}
Let $\Omega$ be a symmetric slice domain. A function $f: \Omega \to \hh$ is regular if, and only if, it is symmetrically analytic.
\end{corollary}

\section{Principal part of a semiregular function}
%Let us recall that regular functions admit a spherical power series expansion (see \cite{spherical}):
%if $g:\Omega\to \HH$ is a regular function, then for any $q_0 \in x_0+y_0\s$ there exists coefficients $A_{n}\in \HH$ such that
%\[g(q)=\sum_{n\ge 0 }((q-x_0)^2+y_0^2)^n[A_{2n}+(q-q_0)A_{2n+1}].\]

\begin{definition}\index{singularity}
Let $f$ be a regular function on a symmetric slice domain $\Omega$. We say that a point $p =x+yI_p \in \hh$ is a \emph{singularity} for $f$ if 
$f_{I_p}:\Omega_{I_p} \to\HH$ has a singularity at $p$. In other words if there exists $R>0$ such that $f$ has the Laurent expansion $f(z) = \sum_{n \in \zz} (z-p)^{n}a_n$  converging for any $z\in L_{I_p}$ with $0<|z|<R$.
\end{definition}

\noindent As proven in \cite{caterina}, if $p=x+yI_p$ is a singularity for a regular function $f$, then $f$ admits a {\em regular Laurent expansion} 
\begin{equation}
f(q) = \sum_{n \in \zz} (q-p)^{*n}a_n,
\end{equation}
converging in $\Sigma(p, R)\setminus \{x+y\s\}$, whose restriction to $L_{I_p}$ coincides with the Laurent expansion of $f_{I_p}$ at $p$. It is clear that, as it happens for regular power series of type \eqref{regularseries}, the domains of convergence of regular Laurent series are not always open sets.  Non-essential singularities are defined as follows.

\begin{definition}\index{singularity!essential}\index{singularity!pole}\index{singularity!removable}\index{order of a pole}
Let $p$ be a singularity for $f$. We say that $p$ is a \emph{removable singularity} if $f$ extends to a neighborhood of $p$ as a regular function. Otherwise consider the expansion
\begin{equation}\label{classification}
f(q) = \sum_{n \in \zz} (q-p)^{*n}a_n.
\end{equation}
We say that $p$ is a \emph{pole} for $f$ if there exists an $m\geq0$ such that $a_{-k} = 0$ for all $k>m$.
%; the minimum such $m$ is called the \emph{order} of the pole and denoted by $ord_f(p)$. If $p$ is not a pole then we call it an \emph{essential singularity} for $f$ and set $ord_f(p) = +\infty$.
\end{definition}
We can now recall the notion of semiregular function, analogue to that of meromorphic function in the complex setting.  
\begin{definition}\index{semiregular function}
A function $f$ is  \emph{semiregular} in a symmetric slice domain $\Omega$ if it is regular in a symmetric slice domain $\Omega'\subseteq\Omega$ such that every point of $\mathcal{S} = \Omega\setminus \Omega'$ is a pole (or a removable singularity) for $f$.
\end{definition}
\noindent If $f$ is semiregular in $\Omega$ then the set $\mathcal{S}$ of its nonremovable poles consists of isolated real points or isolated 2-spheres of type $x+y\s$.

%Poles of regular functions on symmetric slice domains can be either isolated real ponts or isolated two dimensional spheres of the form $x+y\s$. 
The following result shows how we can ``extract'' a pole from a semiregular function, see \cite{caterina}.
\begin{teo}
Let $f:\Omega\to \HH$ be a semiregular function on a symmetric slice domain with a pole at $x_0+y_0\s \subset \Omega$. Then there exist $k\in \nn$ and a unique semiregular function $g$ on $\Omega$, regular on a symmetric slice neighborhood of $x_0+y_0\s$, such that
\[f(q)=((q-x_0)^2+y_0^2)^{-k}g(q).\]
\end{teo}

%\begin{corollary}\label{polefactorization}
%Let $f$ be a semiregular function on a symmetric slice domain $\Omega$. Choose $p = x+yI \in \Omega$, let $m = ord_f(p), n = ord_f(\bar p)$ and suppose $m \leq n$. Then there exists a unique semiregular function $g$ on $\Omega$, without poles in $x+y\s$, such that
%\begin{eqnarray}
%f(q) &=& [(q-p)^{*m}*(q-\bar p)^{*n}]^{-*}* g(q) = \\ 
% &=& \left[(q-x)^2+y^2\right]^{-n} (q-p)^{*(n-m)}* g(q) \nonumber
%\end{eqnarray}
%Furthermore, if $n>0$ then neither $g(p)$ nor $g(\bar p)$ vanish.
%\end{corollary}

\noindent In this case, the {\em spherical order} of the pole is $2k$ at every point of $x_0+y_0\s$ with the possible exception of one single point, where the spherical pole has lesser order.

Using the spherical series expansion \eqref{expansion2} for regular functions we can give the following Definition (see also \cite{Sing}):

\begin{definition} 
Let  $\Omega\subset \HH$ be a symmetric slice domain, let $f:\Omega \to \HH$ be a semiregular function with a pole of spherical order $2k$ at the sphere $x_0+y_0\s$, and let $q_0$ be any point of $x_0+y_0\s$. Then the {\em spherical Laurent series} of $f$ at the sphere $x_0+y_0\s$ is:
\begin{align*}
f(q)&=\sum_{j\ge 0}((q-x_0)^2+y_0^2)^{j-k}[A_{2j}+(q-q_0)A_{2j+1}]\\
&=\sum_{n\ge -k}((q-x_0)^2+y_0^2)^{n}[A_{2(n+k)}+(q-q_0)A_{2(n+k)+1}]
\end{align*}
converging in a symmetric slice open set $U(x_0+y_0\s, R)\setminus \{x_0+y_0\s\}$.
Moreover, the {\em principal part} of $f$ at the spherical pole $x_0+y_0\s$ is defined as
\[P_{x_0+y_0\s}(q)=\sum_{n= 1}^{k}((q-x_0)^2+y_0^2)^{-n}[A_{2(k-n)}+(q-q_0)A_{2(k-n)+1}].\]
\end{definition}

The use of the spherical Laurent series approach  to the Mittag-Leffler Theorem is motivated by the fact that a principal part defined using the apparently simpler regular Laurent series could vary for points of a same spherical pole $x+y\s$.

\section{The Mittag-Leffler Theorem}
We can now prove the announced result, that states that we can find a semi-regular function having prescribed poles and prescribed principal parts.
Denote by $\hat \HH$ the Alexandrov compactification of $\HH$.
\begin{teo}
Let $\Omega \subseteq {\HH} $ be a symmetric slice domain and let $S=\{x_{\alpha}+y_{\alpha}\s\}_{\alpha\in A }$ be a closed and discrete set of two dimensional spheres (or real points) contained in $\Omega$.
For every $\alpha \in A$ let $q_\alpha=x_\alpha+y_\alpha I$, with $I$ any imaginary unit,  $m(\alpha) \in \nn$ and 
\[P_{\alpha}(q)=\sum_{n=1}^{m(\alpha)}((q-x_\alpha)^2+y_\alpha^2)^{-n}[A_{2n}+(q-q_\alpha)A_{2n+1}]\]
with $A_j\in \HH$ for any $j=2, \ldots, 2m(\alpha)+1$.
Then there exists $f$ semiregular on $\Omega$ such that for every $\alpha \in A$ the principal part of $f$ at $x_\alpha+y_\alpha\s$ is $P_{\alpha}(q)$ and such that $f$ does not have other poles in $\Omega$. 
\end{teo}
\begin{proof}
Let $I\in \s$. Thanks to known results in the complex case (see, e.g., Theorem 13.3 in  \cite{rudin}) we can find 
a covering  $\{K^I_n\}_{n\in \nn}$ of $\Omega_I$ such that: $K^I_n$ is a compact set, $K^I_n$ is contained in the interior of $K^I_{n+1}$, every compact subset of $\Omega_I$ is contained in $K^I_n$ for some $n\in \nn$ and every connected component of $\hat{L_I}\setminus K^I_n$ contains a connected component of $\hat{L_I}\setminus \Omega_I$.
The fact that $\Omega$ is a symmetric domain yields that setting, for any $n\in \nn$, $K_n$ to be the symmetrization of $K_n^I$, we obtain a covering  of $\Omega$ such that $K_n$ is a compact set, $K_n$ is contained in the interior of $K_{n+1}$, every compact subset of $\Omega$ is contained in $K_n$ for $n$ sufficiently large, and every connected component of $\hat{\HH}\setminus K_n$ contains a connected component of $\hat{\HH}\setminus \Omega$. 
Moreover, since $\Omega$ is a slice domain, we can suppose that $K_n$ is also slice for any $n\in \nn$.
Let us set 
\[S_1:=S\cap K_1 \quad \text{ and } \quad  S_n:=S\cap(K_{n}\setminus K_{n-1}). \]
The compactness of $K_n$ guarantees that $S_n$ is a finite set of spheres (or real points).
For any $n\in \nn$ define
\[Q_n(q)=\sum_{\alpha \in S_n}P_{\alpha}(q).\]
Notice that, for every $n\in \nn$, $Q_n$ is a rational function, regular on an open neighborhood of $K_{n-1}$.
Thanks to the Runge Theorem for regular functions (see Theorem 4.10 in \cite{runge}), for any $n\in \nn$ we can find a rational function $R_n$ having (prescribed) poles outside $\Omega$ and such that
\begin{equation}\label{unif}
|R_n(q)-Q_n(q)|< 2^{-n} \quad \text{for any } q\in K_{n-1}. 
\end{equation}  
Consider now the semiregular function $f:\Omega\to \HH$ defined by
\[f(q):=Q_1(q)+\sum_{n\ge 2}(Q_n(q)-R_n(q)).\]
We aim to show that $f$ is the desired function.
Fix $N\in \nn$ and split $f$ as
\[f(q)=Q_1(q)+\sum_{n= 2}^N(Q_n(q)-R_n(q))+\sum_{n\ge N+1}(Q_n(q)-R_n(q)).\]
The last term is an infinite sum of functions which are regular in the interior of $K_{N}$. Thanks to equation \eqref{unif}, we get that it converges uniformly to a regular function on the interior of $K_{N}$ (see, e.g., \cite[Remark 3.3]{Sing}).
Hence the function
\[f(q)-Q_1(q)-\sum_{n= 2}^N(Q_n(q)-R_n(q))\]
is regular in the interior of $K_N$ as well, which means that the principal parts of $f$ at the poles contained in $K_N$ are exactly the prescribed $P_{\alpha}(q)$ for $\alpha \in  \bigcup_{n=1}^N S_n$.
Since $N$ was arbitrary, we conclude that $f$ is the desired function.   
\end{proof}

As we already noticed, unlike the case of holomorphic functions, the poles of a regular function over quaternions can be either isolated real points or isolated $2$-spheres of the form $x+y\s$. To conclude, we present { two} simple, and meaningful  { examples} of the Mittag-Leffler phenomenon in the case of semiregular functions. 
{ First we } calculate a semiregular function defined in the entire space of quaternions, such that:
 \begin{enumerate}
\item  its only poles are all the $2$-spheres $n+\s$, centered at $n\in \zz$ with radius $1$;
\item at each such sphere, the principal part is 
 \[
 P_{n+\s}(q) = ((q-n)^2+1)^{-1}
 \]
with minimum possible spherical order equal to $2$.
 \end{enumerate}
Since, for any $N\in \nn$, both 
\[
\sum_{n\ge N+1}\frac{1}{(q-n)^2+1}
\]
and 
\[
\sum_{n\ge N+1}\frac{1}{(q+n)^2+1}
\]
converge uniformly to a regular function inside the open ball centered at the origin and having radius $N$, we get that the function
\[
f(q)= \sum_{n\in \zz}\frac{1}{(q-n)^2+1}
\]
is the desired semiregular function.

{
A second example, peculiar to the quaternionic setting, is that of a semiregular function having infinitely many spheres of poles with spherical order $2$ at each point, except for one point (on every sphere) which has lesser order.  
Namely we want to calculate a semiregular function, defined on the entire space of quaternions, such that:
\begin{enumerate}
	\item  its only poles are all the $2$-spheres $n+\s$, centered at $n\in \zz$ with radius $1$;
	\item at each such sphere, the principal part is 
	\[
	P_{n+\s}(q) = ((q-n)^2+1)^{-1}(q-n-i)
	\]

\end{enumerate}
In this case it is immediate to see that the series
\[\sum_{n \in \zz}P_{n+\s}(q)=\sum_{n \in \zz}\frac{q-n-i}{((q-n)^2+1)}\]
does not converge and hence does not define a semiregular function on $\HH$. However if we sum up the two terms
\[\frac{q-n-i}{((q-n)^2+1)}+ \frac{q+n-i}{((q+n)^2+1)}=\frac{2(q^3-q^2i+q(1-n^2)-(n^2+1)i)}{((q+n)^2+1)((q-n)^2+1)}\]
we get, arguing as in the first example, that
\[\frac{q-i}{q^2+1}+\sum_{n \ge 1}\frac{2(q^3-q^2i+q(1-n^2)-(n^2+1)i)}{((q+n)^2+1)((q-n)^2+1)}\]
 defines the semiregular function we were looking for. 
}

\end{document}